\documentclass[12pt]{article}
\usepackage[centertags]{amsmath}
\usepackage{amsfonts}
\usepackage{amssymb}
\usepackage{latexsym}
\usepackage{amsthm}
\usepackage{newlfont}
\usepackage{graphicx}
\usepackage{listings}
\usepackage{booktabs}
\usepackage{enumerate}
\usepackage{xcolor}
\usepackage[normalem]{ulem}
\newlength{\defbaselineskip}
\newcommand{\setlinespacing}[1]%
           {\setlength{\baselineskip}{#1 \defbaselineskip}}

\newcommand{\N}{{\mathbb{N}}}

\newcommand{\actaqed}{\hfill $\actabox$}
{\medskip\noindent \textit{Proof of #1. }}%
{\actaqed \medskip}

\def\D{{\mathcal D}}

\def\cB{{\mathcal B}}
\def\C{{\mathcal C}}
\def\cC{{\mathcal C}}

\def \Tr{\mathcal T}

\def \cX{\mathcal X}

\def\R{{\mathbb R}}
\def\Z{\mathbb Z}

\def \T{\mathbb T}

\def\bbC{\mathbb C}
\def \<{\langle}
\def\>{\rangle}

\def\Og{\Omega}

\def \ff{\varphi}

\def\bt{\beta}
\def \ro{\varrho}

\def\la{\lambda}

\def \sp{\operatorname{span}}

\def\bx{\mathbf x}
\def\by{\mathbf y}
\def\bz{\mathbf z}
\def\bk{\mathbf k}

\def\bw{\mathbf w}

\def\bF{\mathbf F}

\def\bt{\beta}

\DeclareMathOperator\vol{vol}

\newcommand\LD{\mathcal{LD}}
\newcommand\RD{\mathcal{RD}}
\newcommand\NI{\mathrm{NI}}

\newtheorem{Theorem}{Theorem}[section]
\newtheorem{Lemma}{Lemma}[section]
\newtheorem{Definition}{Definition}[section]
\newtheorem{Proposition}{Proposition}[section]
\newtheorem{Remark}{Remark}[section]
\newtheorem{Example}{Example}[section]
\newtheorem{Corollary}{Corollary}[section]
\numberwithin{equation}{section}

\newtheorem{oldTheorem}{Theorem}[section]

\newtheorem{oldProposition}{Proposition}[section]

\newcommand{\be}{\begin{equation}}
\newcommand{\ee}{\end{equation}}

\begin{document}

\title{One-sided discretization inequalities and sampling recovery}

\author{ I. Limonova, Yu. Malykhin, and   V. Temlyakov 	\footnote{
		 This research  was supported by the Russian Science Foundation (prpject no. 23-71-30001) at Lomonosov Moscow State University  (https://rscf.ru/project/23-71-30001/).
  }}

\newcommand{\Addresses}{{
  \bigskip
  \footnotesize
  \medskip
  I. Limonova, \textsc{ Steklov Institute of Mathematics,\\  Lomonosov Moscow State University,\\ and Moscow Center for Fundamental and Applied Mathematics.\\ 
E-mail:} \texttt{limonova\_irina@rambler.ru}

\medskip
Yu. Malykhin, \textsc{ Steklov Institute of Mathematics,\\  Lomonosov Moscow State University.
\\
E-mail:} \texttt{malykhin@mi-ras.ru}

 \medskip
  V.N. Temlyakov, \textsc{ Steklov Institute of Mathematics,\\  Lomonosov Moscow State University,\\ and Moscow Center for Fundamental and Applied Mathematics,\\ University of South Carolina.
  \\
E-mail:} \texttt{temlyak@math.sc.edu}

}}
\maketitle

\begin{abstract}
	{Recently, in a number of papers it was understood that results on sampling discretization and on the universal sampling discretization can be successfully used in the problem of sampling recovery. Moreover, it turns out that it is sufficient to only have a one-sided discretization inequality for some of those applications. This motivates us to write the present paper as a survey paper, which includes new results, with the focus on the one-sided discretization inequalities and their applications in the sampling recovery. In this sense the paper complements the two existing survey papers on sampling discretization. }
\end{abstract}


{\it MSC classification 2000:} Primary 65J05.

\section{Introduction} 
\label{I}

There has been significant progress in the study of sampling discretization of integral norms for  both  a designated  finite-dimensional function space and   a finite collection of such function spaces (universal discretization).  Sampling discretization results  turn out to be very useful  in various applications, particularly in  sampling recovery. For recent results 
on sampling recovery we refer the reader to
\cite{CM}, \cite{KU}, \cite{KU2},
\cite{NSU}, \cite{VT183}, \cite{CoDo}, \cite{KUV}, \cite{TU1}, \cite{LimT},
\cite{JUV}, \cite{DKU},  \cite{DTM1}, \cite{DTM2}, \cite{DTM3}, \cite{BSU}, \cite{KPUU},
\cite{KPUU2}, \cite{VT2023}, \cite{KT24}, \cite{T24}.

The systematic study of sampling discretization of the $L_p$-norms of functions from finite-dimensional subspaces has begun in 
2017 in \cite{VT158} and \cite{VT159}. First results in this direction were obtained in the 1930s and go back to S.~Bernstein, J.~Marcinkiewicz, and A.~Zygmund. At the moment, this area has developed into a vast and actively growing field of research with deep connections with other important areas of research. There are two survey papers  \cite{DPTT} and \cite{KKLT} on the topic. Paper \cite{DPTT} covers results on exact weighted discretization, discretization of the uniform norm of the multivariate trigonometric polynomials, and some results on the universal discretization. Paper \cite{KKLT} covers recent results on sampling discretization and discusses in detail their connections with other areas of research, in particular, with the theory of moments of random vectors, the theory of submatrices  {with good properties}, embedding of finite-dimensional subspaces, sparse approximation, learning theory, and  sampling recovery.
Recently, in a number of papers it was understood that results on sampling discretization and on the universal sampling discretization can be successfully used in the problem of sampling recovery. Moreover, it turns out that it is sufficient to only have a one-sided discretization inequality for some of those applications. This motivates us to write the present paper on the one-sided discretization inequalities and their applications in the sampling recovery. In this sense the paper complements the survey papers \cite{DPTT} and \cite{KKLT}.

Let $(\Omega,\mu)$ be a probability space. 
We consider measurable functions that are defined at each point of $\Omega$ and we
do not identify functions that differ on a set of zero measure. (See also~\cite[Section~2]{KKLT}, for a
justification.) By the
$L_p$, $1\le p< \infty$, norm we understand

$$
\|f\|_p:=\|f\|_{L_p(\Omega,\mu)} := \left(\int_\Omega |f|^p\,d\mu\right)^{1/p}.
$$

By the $L_\infty$-norm we understand the uniform norm of bounded functions
$$
\|f\|_\infty := \sup_{\omega\in\Omega} |f(\omega)|
$$
and with some abuse of notation we occasionally write $L_\infty(\Omega)$ for the space $\cB(\Omega)$ of bounded functions on $\Omega$.
In Sections \ref{B} and \ref{UD} it is convenient for us to
assume that $\Omega$ is a compact subset of $\R^d$ and to consider the space
$\cC(\Omega)$ of continuous functions instead.

We also consider the discrete space $L_p^m$ of  vectors $\bx=(x_1,\dots,
x_m)\in\mathbb{R}^m$ (or $\mathbb C^m$) with the norm
\begin{equation}
    \label{discrete_norm}
\|\bx\|_p := \begin{cases}
\left(\frac{1}{m}\sum_{j=1}^m|x_j|^p\right)^{1/p},&\quad 1\le p<\infty,\\
\max\limits_{1\le j\le m}|x_j|,&\quad p=\infty.
\end{cases}
\end{equation}
This norm is less frequently used than the $\ell_p^m$-norm (without weights $1/m$) but it is more convenient
in our context.

Given a sequence of points $\xi^1,\ldots,\xi^m\in\Omega$, we can
associate with any function $f$ defined on $\Omega$ a sampling vector
\begin{equation}\label{sampling}
S(f,\xi) := (f(\xi^1),\ldots,f(\xi^m)).
\end{equation}

The previous papers on discretization were focused on two-sided
inequalities, which establish that the discrete norm
of a sample vector is bounded from below and from above by the integral $L_p$-norm of a function (multiplied by some constants). These kind of inequalities
are also known under the name of Marcinkiewicz-Zygmund inequalities.
Let us describe them in more detail.

\paragraph{The Marcinkiewicz-type discretization problem.}
Let $(\Omega,\mu)$ be a probability space. We say that a linear subspace $X_N$ (index $N$
here, usually, stands for the dimension of $X_N$) of $L_p(\Omega,\mu)$, $1\le p
< \infty$, admits the Marcinkiewicz-type discretization theorem with
parameters $m\in \N$ and $p$ and positive constants $C_1\le C_2$ if there exists
a set $\{\xi^j\}_{j=1}^m\subset\Omega$ such that for any $f\in X_N$ we have
\be\label{I1}
C_1\|f\|_p^p \le \frac{1}{m} \sum_{j=1}^m |f(\xi^j)|^p \le C_2\|f\|_p^p. 
\ee
The following problem was considered: For which $m$ a given subspace $X_N$
does admit the Marcinkiewicz-type discretization inequality~\eqref{I1}?

\paragraph{The Bernstein-type discretization problem.}
In the case $p=\infty$ we define $L_\infty$ as the space of bounded functions on $\Omega$  and ask for
\begin{equation}\label{I2}
C_1\|f\|_\infty \le \max_{1\le j\le m} |f(\xi^j)| \le  \|f\|_\infty.
\end{equation}

Recently it was established that results on sampling discretization and on
universal sampling discretization of the square norm are useful in sampling
recovery (see \cite{VT183}) and in sparse sampling recovery (see
\cite{DTM1}--\cite{DTM2}) with error measured in the square norm. It was
discovered that in some cases it is sufficient to impose a one-sided sampling
discretization assumption for proving the corresponding sampling recovery
results.

Inequality (\ref{I1}) consists of two one-sided inequalities: The Left
Discretization Inequality (LDI) and the Right Discretization Inequality (RDI).
In this paper we concentrate separately on the LDI and on the RDI. Clearly,
the Bernstein-type discretization problem is an example of the one-sided
discretization inequality, namely, the LDI. The RDI in the inequality
(\ref{I2}) is trivial. Note, that there might be other names for the one-sided
inequalities in the literature. For instance,  {D.~}Lubinsky calls RDI -- {\it
forward inequality} and LDI -- {\it converse inequality}  {(see, e.g., \cite{Lub})}.  

In this paper we focus on more general one-sided
discretization inequalities
between the integral $L_p$-norm of a function $f$ and the discrete $L_q^m$-norm
(or weighted $\ell_q$-norm) of the vector $S(f,\xi)$. We stress that the parameters $p$ and
$q$ may be different.

We consider the following settings.

\paragraph{One-sided discretization problems.}

Given a probability space $(\Omega,\mu)$, a set $X$ of functions defined on
$\Omega$, and parameters $1\le p,q\le \infty$, $D>0$,
$m\in\N$, we are interested in the
following properties.

LDI. We say that a set $X$ admits a Left Discretization Inequality if
\be\label{I1L}
\|f\|_p \le D \left(\frac{1}{m} \sum_{j=1}^m |f(\xi^j)|^q\right)^{1/q},\quad
\forall f \in X
\ee
for some sequence of points $\xi^1,\ldots,\xi^m\in\Omega$. We denote this
property by $X \in \LD(m,p,q,D)$.

RDI. We say that a set $X$ admits a Right Discretization Inequality if
\be\label{I1R}
  \left(\frac{1}{m} \sum_{j=1}^m |f(\xi^j)|^q \right)^{1/q} \le D\|f\|_p, \quad
  \forall f \in X
\ee
for some sequence of points $\xi^1,\ldots,\xi^m\in\Omega$. We denote this
property by $X \in \RD(m,p,q,D)$.

(In the case $q=\infty$ the discrete norm in~\eqref{I1L} or~\eqref{I1R} is
understood as usual, see~\eqref{discrete_norm}.)

Note that in the definitions of LDI and RDI we allow some of the points
$\xi^1,\ldots,\xi^m$ to coincide. Allowing some liberty, we will write $\xi =
\{\xi^j\}_{j=1}^m$, bearing in mind that some elements of this set may coincide. 

The problem we consider here is to find the smallest $m\in\N$ such that LDI/RDI hold
true for a given $N$-dimensional subspace $X_N$ or for a collection of subspaces $X_N$.

One can formulate LDI/RDI in terms of the sampling operator $S(\cdot,\xi)$
acting from $L_p$ to $L_q^m$.
The RDI is just the norm bound $\|S f\|_q \le D\|f\|_p$ for all  $f\in X$ and the LDI is the bound $\|f\|_p \le D\|S f\|_q$  for all  $f\in X$.

We may write RDI($p,q$) or LDI($p,q$) if we want to emphasize the parameters.
For the sake of brevity in the case $p=q$ we often write just one parameter,
e.g. $\LD(m,p,D)$, RDI($p$).
 
The above inequalities (\ref{I1}), (\ref{I1L}), and (\ref{I1R}) are formulated
in the case of discretization with equal weights -- weight $1/m$ for each
sample point $\xi^j$ {, $j=1,\dots,m$}. In the case of general weights (weight $\la_j$ for the
sample point $\xi^j$, $j=1,\dots,m$) we obtain analogues of LDI and RDI for
weighted discretization, which we call WLDI:
$$
\|f\|_p \le D\left(\sum_{j=1}^m \lambda_j|f(\xi^j)|^q\right)^{1/q}
\quad\forall f\in X,
$$
and WRDI:
$$
\left(\sum_{j=1}^m \lambda_j|f(\xi^j)|^q\right)^{1/q}\le D\|f\|_p
\quad\forall f\in X.
$$

The $\LD$ and $\RD$ properties are almost monotone in the number
of sampling points in the following sense. If $X\in\LD(m,p,q,D)$ then
$X\in\LD(km,p,q,D)$ for any
$k\in\mathbb N$ (repeat each point $k$ times); moreover,
$X\in\LD(m',p,q,2^{1/q}D)$ for all $m'>m$.
Indeed, if $m'\le 2m$, then consider $m$ points ensuring that
$X\in\LD(m,p,q,D)$ and repeat $m'-m$ of them. Then the $\ell_q^m$-norms of
functions do not decrease and the normalizing factor $1/m^{1/q}$ decreases with
the coefficient $(m'/m)^{1/q}\le 2^{1/q}$; we obtain $X\in\LD(m',p,q,2^{1/q}D)$.
The general case reduces to this in view of the first observation.

The same holds for RDI. Hence if we
do not care about multiplicative absolute constants there is no difference
between, say, the property
``$X\in\LD(m,p,q,C_2)$ for some $m\le C_1N\log N$''
and ``$X\in\LD(C_1N\log N,p,q,C_2)$'' (in the last expression we even allow
non-integer first parameter $C_1N\log N$ and assume that the number of points is $[C_1N\log N]$).
 
\paragraph{The structure of the paper.}
We now give a brief description of our results. In Section~\ref{RI} we
mostly study the Right Discretization Inequality. We focus on the
lower bounds for $m$ such that $\RD(m,p, q, D)$ is possible. In Proposition~\ref{RIP1} we discuss the case $2<p<\infty$, $1\leq q<\infty$ and obtain a lower bound on $m$ for
the RDI to hold under certain assumptions on a subspace $X_N$. As a corollary
of this lower bound (see Corollary \ref{RIC1}) we establish that for a subspace
$\Tr(\Lambda_N)$ spanned by a lacunary trigonometric system of size $N$ to have
the RDI in the case $2<q<\infty$ it is necessary to use at least $N^{q/2}$ points
in the sense of order. We show (see Remark~\ref{RemLosi}) that for the
LDI to hold for the subspace $\Tr(\Lambda_N)$ it is sufficient to use $m$ of
order $N$ points. Also, we prove (see Proposition~\ref{RIP3}) that if the sampling operator is injective, then there is no universal bound on $m$ in terms of $N$, which guarantees RDI for all
$N$-dimensional subspaces. 

In Section~\ref{LI} we discuss LDI. We show that a weighted LDI with nonnegative
weights implies an LDI (see Proposition~\ref{AP4}). As a consequence we get Proposition~\ref{AP6} which shows that
for any $N$-dimensional subspace the LDI($2$) holds with $m$ of order $N$ points. We also formulate an open problem.
 
In Section~\ref{M} we discuss a connection between sampling discretization
problems and known settings on finding submatrices of small size with certain
good properties. In particular, we explain that the RDI(2) with $m=N$ points
is a corollary of the A.\,A.~Lunin's result on matrices. 

In Section~\ref{B} we discuss the problem of sampling recovery. Mostly, it is a
survey of known results with simple comments on their improvements.
Here the case $p\ne q$ is important. For instance,
it turns out that some of the sampling recovery results, which were proved
using LDI($p,p$), can be proved under a weaker assumption
LDI($p,\infty$):
\be\label{A17a}
\|f\|_p \le D \max_{1\le j \le m} |f(\xi^j)|,
\ee
where $p\in[1,\infty)$, see Theorem~\ref{BT2}.
Also, we show in Section~\ref{B} how known results on discretization of the
uniform norm can be combined with a result of J.~Kiefer and
J.~Wolfowitz~\cite{KW} to obtain results on LDI.

In Section~\ref{UD} we give a brief survey of the very recent results from 
\cite{DTM1}--\cite{DTM3}. In addition, in the same way as in Section~\ref{B} we demonstrate 
that some of those results can be improved by replacing the LDI  with parameter $p$ assumption by the weaker LDI($p$,$\infty$), $p\ge 1$, see \eqref{A17a}.

 In Section~\ref{C} we discuss sampling discretization of the uniform norm. By its nature, 
 sampling discretization of the uniform norm is the LDI problem because the corresponding 
 RDI is trivial in this case. In Section~\ref{C} we focus on the lower bounds on $m$ for the sampling discretization result to hold. Theorem~\ref{CT1} is a new result, which extends the
 previously known result for the trigonometric system to the case of more general systems. 
 Then, we discuss an application of Theorem~\ref{CT1} to the case, when the system of interest is a Sidon system. 

We often use the concept of Nikol'skii inequality. For the reader's convenience we formulate it here. Other definitions and notations are introduced below in the text.

{\bf Nikol'skii-type inequalities.} Let $1\le p\le q\le\infty$ and $X_N\subset L_q(\Omega, \mu)$. The inequality
\begin{equation}\label{I4}
\|f\|_q \leq M\|f\|_p,\   \ \forall f\in X_N
\end{equation}
is called the Nikol'skii inequality for the pair $(p,q)$ with the constant $M$. 
We will also use the brief form of this fact: $X_N \in \NI(p,q,M)$. Typically, $M$ depends on $N$, for instance, $M$ can be of order $N^{\frac{1}{p}-\frac{1}{q}}$.

\section{Right Discretization Inequalities (RDI)}
\label{RI}

In this section we mostly discuss necessary conditions on $m$ for RDI and WRDI to hold. 
We begin with the weighted discretization.

\begin{Lemma}\label{RIL1}
    Let $2<p<\infty$, $1\leq q<\infty$ and let $X_N\in\NI(2,p,M)$ with some
    constant $M$.  Assume that
    $\{\xi^j\}_{j=1}^m\subset \Omega$ and nonnegative weights $\{\la_j\}_{j=1}^m$
    are such that for any $f\in X_N$ we have
\be\label{RI1}
    \left( \sum_{j=1}^m\la_j |f(\xi^j)|^q \right)^{1/q} \le D\|f\|_p.
\ee
    Then for any orthonormal basis $\{u_i\}_{i=1}^N$ of $X_N$ and any $j\in
    \{1,\ldots,m\}$ we have
$$
    \la_j\left( \sum_{i=1}^N |u_i(\xi^j)|^2\right)^{q/2} \le (DM)^q.
$$
\end{Lemma}
\begin{proof} It is well known and easy to check that for any $\omega\in\Omega$ we have
\be\label{RI3}
\left(\sum_{i=1}^N |u_i(\omega)|^2\right)^{1/2} = \sup_{f\in X_N:\,\|f\|_2\le 1}
    |f(\omega)|.
\ee
Using the trivial bound $\la_j|f(\xi^j)|^q \le \sum_{k=1}^m\la_k |f(\xi^k)|^q$, we obtain from (\ref{RI3})
$$
\la_j\left(\sum_{i=1}^N |u_i(\xi^j)|^2\right)^{q/2} \le \sup_{f\in X_N:\,\|f\|_2\le 1}  \sum_{k=1}^m \la_k|f(\xi^k)|^q.
$$
By (\ref{RI1}) and by our assumption on the Nikol'skii inequality we continue
$$
    \le \sup_{f\in X_N:\,\|f\|_2\le 1} (D\|f\|_p)^q \le (DM)^q.
$$
This completes the proof of Lemma \ref{RIL1}.
\end{proof}

\begin{Remark}\label{R_negative_weights}
In Lemma~\ref{RIL1} we assume that all weights are positive. But discretization
with negative weights also makes sense. For instance, it was noted in the paper
\cite{Lim21} that even for $p=2$ there exist subspaces such that for the
minimal number $m$ of points providing for any $f\in X_N$ the equality
$$
\|f\|_p^p = \sum_{j=1}^m \lambda_j |f(\xi^{j})|^p
$$
(exact discretization) we need some $\lambda_j$, $j=1,\dots, m$, to be negative.
\end{Remark}

The following Proposition~\ref{RIP1} is a direct corollary of Lemma~\ref{RIL1}. 

\begin{Proposition}\label{RIP1}
    Let $2<p<\infty$, $1\le q<\infty$ and let $X_N\in\NI(2,p,M)$ with some
    constant $M$. Suppose that $X_N$ has an
    orthonormal basis $\{u_i\}_{i=1}^N$ with the property
    $\sum_{i=1}^N|u_i(\omega)|^2\ge cN$, $c>0$, for all $\omega\in \Omega$. Assume that $X_N~\in~\RD(m,p,q,D)$. Then
$$
    (cN)^{q/2} \le m (DM)^q.
$$
\end{Proposition}

\begin{Corollary}\label{RIC1}
    Let $\Lambda_N = \{k_i\}_{i=1}^N$ be a lacunary sequence: $k_{i+1}~\ge~bk_i$, $b>1$, $i=1,\dots,N-1$. Denote
$$
\Tr(\Lambda_N):= \left\{f\,:\, f(x) = \sum_{k\in \Lambda_N} c_ke^{ikx},\quad x\in\T \right\}.  
$$
    Assume that for $2<p<\infty$, $2< q<\infty$ we have
    $\Tr(\Lambda_N)\in\RD(m,p,q,D)$.
    Then, as $\Tr(\Lambda_N)\in\NI(2,p,M)$, where $M=M(b,p)$ (see, for
    instance,~\cite{Zyg}, Theorem 8.20), we have
$$
    N^{q/2} \le m C(p,q,b) D^q.
$$
\end{Corollary}

\begin{Remark}\label{RemLosi}
    In contrast to the RDI case, the LDI with $2<p<\infty$, $2<q<\infty$ for lacunary polynomials may
    be satisfied using $m=O(N)$ points. Indeed, by Theorem~1.1 from~\cite{VT158} there are $m \le CN$ points
    $\xi^1,\ldots,\xi^m$ that provide $L_2$-discretization
    and this immediately implies LDI:
    $$
    \|f\|_p \le C_1\|f\|_2
    \le C_2 \left(\frac{1}{m}\sum_{j=1}^m |f(\xi^j)|^2\right)^{1/2} 
    \le C_2 \left(\frac{1}{m}\sum_{j=1}^m |f(\xi^j)|^q\right)^{1/q}.
    $$
    
\end{Remark}

\begin{Proposition}\label{RIP2}
    Let $2<p<\infty$ and let $X_N\in\NI(2,p,M)$ with some
    constant $M$.
    Assume that $\{\xi^j\}_{j=1}^m\subset \Omega$ and nonnegative weights $\{\la_j\}_{j=1}^m$ are such that for any $f\in X_N$ we have
\be\label{RI8}
    D_1^{-1}\|f\|_2 \le \left(\sum_{j=1}^m \la_j |f(\xi^j)|^p \right)^{1/p} \le D_2\|f\|_p.
\ee
Then  
$$
    N^{p/2} \le m (K_p D_1 D_2 M)^p,
$$
where $K_p$ is a constant from the Khinchin inequality.
\end{Proposition}

\begin{proof}
    Let $\{u_i\}_{i=1}^N$ be an orthonormal basis of $X_N$. Consider the functions
$$
f(\omega,\theta) := \sum_{i=1}^N r_i(\theta)u_i(\omega),
$$
where $\{r_i(\theta)\}_{i=1}^N$ are the Rademacher functions defined on $[0,1]$. Then by our assumption 
(\ref{RI8}) we obtain for each $\theta\in [0,1]$
\be\label{RI10}
    D_1^{-p} N^{p/2} = D_1^{-p} \|f(\cdot,\theta)\|_2^p \le \sum_{j=1}^m \la_j |f(\xi^j,\theta)|^p.
\ee
Integrating~(\ref{RI10}) with respect to $\theta$ and using the Khinchin inequality we get
$$
    D_1^{-p} N^{p/2} \le \int_0^1\sum_{j=1}^m \la_j |f(\xi^j,\theta)|^p\,d\theta \le
    K_p^p\sum_{j=1}^m\la_j\left(\sum_{i=1}^N |u_i(\xi^j)|^2\right)^{p/2}.
$$
Applying Lemma \ref{RIL1} with $q=p$ we find
$$
    D_1^{-p}N^{p/2} \le K_p^p m D_2^p M^p.
$$
\end{proof}

\begin{Remark}\label{RIR1} In the case $2<p<\infty$ the condition
$$
    \|f\|_2\le D_1 \left(\frac{1}{m}\sum_{j=1}^m |f(\xi^j)|^p\right)^{1/p}
$$
is weaker than the two related conditions
$$
    \|f\|_p \le D_1\left(\frac{1}{m}\sum_{j=1}^m |f(\xi^j)|^p\right)^{1/p} \quad\text{and}\quad 
    \|f\|_2 \le D_1\left(\frac{1}{m}\sum_{j=1}^m |f(\xi^j)|^2\right)^{1/2}.
$$
\end{Remark}

\paragraph{RDI with injective sampling operator.}
Now we make a simple observation that in the case of an injective sampling operator
$f\mapsto S(f,\xi)$ there is no \textit{a priori} upper bound for the number of points
$m$ that provide RDI.

Given $a\in(0,1/2]$, we define the function $f_a\in C[0,1]$ as
follows:
$f_a(x) := 1$, $0\le x\le a$; $f_a(x):=2-x/a$, $a\le x\le 2a$;
$f_a(x):=0$, $2a \le x\le 1$.

\begin{Proposition}\label{RIP3} 
    {Let $p,q\in[1,\infty)$ and let $X_N$ be a linear subspace
    of $C[0,1]$ that contains functions $f_a$ and $f_{a/2}$ for some
    $a\in(0,1/2]$.}
Suppose that a set $\{\xi^j\}_{j=1}^m\subset\Omega$ provides {the property
    $X_N\in\RD(m,p,q,D)$ with injective sampling operator, i.e.}
\begin{equation*}
    0 < \left(\frac{1}{m}\sum_{j=1}^m|f(\xi^j)|^q\right)^{1/q} \le
    D\|f\|_p,\quad \forall f\in X_N{\setminus\{0\}}.
\end{equation*}
    Then
    \be
    \label{f_a_bound}
    {
    m\ge D^{-q}(2a)^{-q/p}.
}
    \ee
\end{Proposition}

\begin{proof}
    {
    We apply the injectivity of the sampling operator to the function $f_{a/2}$
    and obtain that at least one among the $\xi^j$, $j=1,\dots, m$, lies in the segment $[0,a]$. Therefore, the RDI($p,q$) for $f_a$ implies
    $$
    \frac{1}{m^{1/q}} \le \left(\frac1m\sum_{j=1}^m|f_a(\xi^j)|^q\right)^{1/q}
    \le D\|f_a\|_p \le D(2a)^{1/p}.
    $$
    The required inequality follows.}
\end{proof}

{We can take $a$ in~\eqref{f_a_bound} to be arbitrarily
small and obtain arbitrarily large lower bounds on $m$.}
Hence there is also no upper bound for the number of points $m$ that provide
Marcinkiewicz-type discretization~\eqref{I1}.
 Note that Proposition~\ref{RIP3} gives the solution
to Open Problem 1 from \cite{KKLT}. In terms of \cite{KKLT} it means that for
each $1\le p<\infty$ and any positive constants $C_1$ and $C_2$ we have
$sd(N,p,C_1,C_2) =\infty$.

Note that any subspace $X_N$ admits RDI(2,2)
with $m=N$ points (see Proposition 4.1). This means that the
injectivity condition for RDI is essential.
The LDI(2,2) also holds for any $X_N$ with $m=O(N)$
points (see Proposition 3.2). However, we cannot guarantee
both LDI and RDI with a common set of points
of some size bounded by $m(N)$. Otherwise, we would
obtain RDI with an injective sampling operator and
this would contradict Proposition 2.3.

The following corollary of Proposition~$3.2$ from the paper
\cite{FG} gives the lower bounds on the number of points for the RDI
for subspaces with the Nikol'skii inequality. We refer the reader to \cite{FG} for the
construction and for the result in the case $X_N\in NI(1,\infty, (1+\varepsilon)N)$.
\begin{oldTheorem}[\cite{FG}]
  For any $1\le q<2$ and $N\in\mathbb{N}$ there exists an $N$-dimensional subspace $X_N\subset L_1([0,1])$ with the following properties: 
  \\(a) $X_N\in NI(q,\infty, K_q^{-1}N^{1/q})$;
  \\(b) If $m\in\mathbb{N}$ and $t_j, j=1,\dots, m$, are such that $\sum\limits_{j=1}^m|f(t_j)|^q>0$ for all $f\in X_N\backslash \{0\}$, then there exists $f\in X_N$ with
  \be\label{LI_F_q}
   \frac{1}{m}\sum\limits_{j=1}^m|f(t_j)|^q\ge \frac{N^{2-q/2}}{m}\|f\|_{L_q}^q.
  \ee
  In particular, it follows from \eqref{LI_F_q} that the RDI($q$) with injective
sampling operator is not true if $m=o(N^{2-q/2})$.
\end{oldTheorem}

{\bf A comment on WRDI.}
Let $2\leq r<p<\infty$, $X_N\in\NI(2,p,M)$, points $\xi^1,\dots,\xi^m\in\Omega$
and weights $\lambda_1,\dots,\lambda_m\in\R_+$, $\lambda_1+\dots+\lambda_m=1$, be such that  
\begin{equation*}
\left(\sum_{j=1}^{m} \lambda_j|f(\xi^j)|^p\right)^{1/p}\le D \|f\|_p, \quad  \forall f\in X_N.
    \end{equation*}
Then, writing $\lambda_j = \lambda_j^{1-r/p}\lambda_j^{r/p}$, by the H{\"o}lder inequality we obtain
$$
\sum_{j=1}^{m} \lambda_j|f(\xi^j)|^r\le \left(\sum_{j=1}^{m} \lambda_j|f(\xi^j)|^p\right)^{\frac{r}{p}}
$$
\begin{equation*}
    \leq D^r\|f\|_p^r\le D^r M^r\|f\|_2^r\le D^r M^r\|f\|_r^r.
\end{equation*}
Thus, $\NI(2,p, M)$ and WRDI($p$) with a constant $D$ imply WRDI($r$) with the constant $DM$.

\section{Left Discretization Inequalities (LDI)}\label{LI}

We will prove that a WLDI with nonnegative weights implies an LDI if the sum of weights is
bounded from above.
This extra property of weights is true, for example, when we have the two--sided weighted
discretization. We begin with a slightly more general observation than the one from \cite{VT183}.

\begin{Remark}\label{rem_marcink_sum_weights}
    Let $\mathcal F$ be a collection of finite--dimensional subspaces such that if
    $X_N\in\mathcal F$ then $X_N' := \{f=g+c\colon g\in X_N,\, c\in \mathbb
    R\}\in\mathcal F$.
    Suppose that for some $1\le p_1\le p_2<\infty$ any $N$-dimensional $X_N\in\mathcal F$ admits the discretization
    \be\label{wmdi}
    D_1^{-1}\|f\|_{p_1} \le \left(\sum_{j=1}^m \lambda_j
    |f(\xi^j)|^q\right)^{1/q} \leq  D_2 \|f\|_{p_2},\  \ \forall f\in X_N,
    \ee
    with $m\le m(N)$ points and some nonnegative weights $\{\lambda_j\}_{j=1}^m$. Then
    any $X_N\in \mathcal F$ satisfies~\eqref{wmdi} with
    $m'\le m(N+1)$ points 
    and nonnegative weights $\{\lambda'_j\}_{j=1}^{m'}$ such that $\sum \lambda_j' \le D_2^q$.
\end{Remark}

\begin{Proposition}\label{AP4}  Let $X$ be a set of functions in $L_p(\Omega,\mu)$,  points $\xi^1,\dots,\xi^m\in\Omega$
    and nonnegative weights $\lambda_1,\dots,\lambda_m$ be such that
    $\lambda_1+\dots+\lambda_m\le C$ and 
    \begin{equation*}
    \|f\|_p\le D\left(\sum_{j=1}^{m} \lambda_j|f(\xi^j)|^q\right)^{1/q}, \quad  \forall f\in X.
\end{equation*}
    Then there exists $m_0\le (C^2+1)m$ such that $X\in \LD(m_0,p,q,D(C~+~C^{-1})^{1/q})$.
\end{Proposition}
\begin{proof}
    Denote $m_0:=\sum\limits_{i=1}^m ([\lambda_iCm]+1)$, then $m_0\le \sum\limits_{i=1}^m \lambda_iCm+m\le (C^2+1)m$. We will construct a new set of points. For any $j=1,\dots, m$ we take the point $\xi^j$ $[\lambda_jCm]+1$ times. Denote all these points as $\xi_0^j$, $j=1,\dots, m_0$. Then
    \begin{gather*}
    \left(\sum_{j=1}^{m_0} \frac{1}{m_0}|f(\xi_0^j)|^q\right)^{1/q}=\left(\sum_{j=1}^{m} \frac{([\lambda_jCm]+1)}{m_0}|f({\xi}^j)|^q\right)^{1/q}\ge \\
    \ge \left(\sum_{j=1}^{m} \frac{\lambda_jCm}{(C^2+1)m}|f({\xi}^j)|^q\right)^{1/q}\ge \|f\|_p \left(\frac{C}{C^2+1}\right)^{1/q}/D.
    \end{gather*}
\end{proof}

Let us apply this technique to some known theorems on discretization.

\begin{Proposition}\label{AP6}
Let $1\leq p<\infty$. There exist positive constants $C_1(p)$ and $C_2(p)$ such that 
    each $N$-dimensional subspace $X_N \subset  L_p(\Omega,\mu)$ satisfies the
    property $X_N\in\LD(m, p, C_2)$, i.e. there is a set of points $\{\xi^j\}_{j=1}^m$ such that
\begin{equation*}
\|f\|_p \le C_2  \left(\sum_{j=1}^{m} \frac{1}{m}|f({\xi}^j)|^p\right)^{1/p},\quad \forall f\in X_N,
\end{equation*}
        where
    \begin{itemize}
        \item $m\le C_1N$, if $p=2$;
        \item $m\le C_1N\log{N}$, if $p=1$;
        \item $m\le C_1N\log{N}(\log\log{N})^2$, if $1<p<2$;
        \item $m\le C_1N^{p/2}$ if $p$ is an even integer;
        \item $m\le C_1N^{p/2}\log N$ for $p>2$.
    \end{itemize}
\end{Proposition}

\begin{proof}
    In the case $p=2$ we use the following result
    \cite[Theorem 6.4]{DPSTT2}: there exist three absolute constants $C_1'$, $c_0$, $C_0$ and a set of
    $m\leq C_1'N$ points $\xi^1,\ldots, \xi^m\in\Omega$, and a set of
    nonnegative  weights $\lambda_j$, $j=1,\ldots, m$, such that
    $$
    c_0\|f\|_2^2\leq  \sum_{j=1}^m \lambda_j |f(\xi^j)|^2 \leq C_0
    \|f\|_2^2,
    \quad \forall f\in X_N.
    $$
    This WLDI, Remark~\ref{rem_marcink_sum_weights} and Proposition~\ref{AP4}
    imply the required LDI.
    
    Other cases are also covered by known discretization results,
    see~\cite[\textbf{D.20. An upper bound}]{KKLT} for $p>2$
    and~\cite[Corollaries~1.1, 1.2]{DKT} for $1~\le~p~<~2$.
\end{proof}

\begin{Remark}\label{not_continuous}Note that in the cited papers the discretization results
    are often formulated for subspaces of continuous
    functions, $X_N\subset C(\Omega)$.
    However, it is always possible to properly discretize any subspace
    $X_N\subset L_p$ using an appropriate amount of points
    (see, e.g., Proposition 4.A)
    and pass to a subspace of functions defined on a
    discrete set. These functions are continuous by definition.
    So, cited results work in the general setting.
    \end{Remark}

We discovered that the case $p=2$, i.e. the LDI(2,2) 
\be\label{ldi_2_2}
\|f\|_2 
\le C_2  \left(\sum_{j=1}^{m} \frac{1}{m}|f({\xi}^j)|^2\right)^{1/2},
\quad \forall f\in X_N,
\quad m\le C_1N,
\ee
was obtained independently by a different method in the
paper~\cite[Theorem~6.2]{BSU}, see also~\cite[Lemma~11]{KPUU2}.

{\bf Open problem \ref{LI}.1.}
Under what condition on $m$ does each $N$-dimensional subspace $X_N \subset
\cC(\Omega)$ have the property $X_N\in\LD(m,p,\infty,D)$, where $2<p<\infty$
and $D=D(p)$:
$$
\|f\|_p \le D \max_{1\le j \le m} |f(\xi^j)|,\quad\forall f\in X_N,
$$
with an appropriate set of points $\{\xi^j\}_{j=1}^m$?

{\bf Comment \ref{LI}.1.} Theorem~1.3 from \cite{KKT} on discretization of the
uniform norm imply that in Open problem~\ref{LI}.1 we can always take $m=9^N$,
which means an exponential growth in $N$. Theorem~1.2 from \cite{KKT} shows that for the subspace $\Tr(\Lambda_N)$ (see Corollary~\ref{RIC1}) we need $(N/e)e^{CN/D^2}$ points for the LDI($\infty$) to hold.  Proposition~\ref{AP6} for $p=2$ guarantees that we
can take $m$ of order $N$ in the case $p\leq 2$. From Proposition~\ref{AP6} it
follows that we can take $m\le C N^{p/2}\log{N}$ in the case $p>2$ and $m\le
CN^{p/2}$ for even $p$.

The following theorem from \cite{KPUU2} is the LDI($p$,$2$).
\begin{Theorem}[{\cite[ {Theorem~13}]{KPUU2}}]\label{LDI_p_2}
    Let $2\leq p\leq \infty$. There exist positive constants $C_1=4$ and $C_2=83$ such that 
    each $N$-dimensional subspace $X_N \subset \cC(\Omega)$ satisfies the
    property $X_N\in\LD(C_1N, p, 2, C_2 N^{\frac{1}{2}-\frac{1}{p}})$, that is, there is a set of $m\leq C_1N$ points $\{\xi^j\}_{j=1}^m$ such that
\begin{equation*}
\|f\|_p \le C_2 N^{\frac{1}{2}-\frac{1}{p}} \left(\sum_{j=1}^{m} \frac{1}{m}|f({\xi}^j)|^2\right)^{1/2},\quad \forall f\in X_N.
\end{equation*}
\end{Theorem}

See also some results on LDI in Section~\ref{B}: {Theorem~\ref{TrDi},} Theorem~\ref{BP1},
{Theorem~\ref{BP1a}, and  Theorem~\ref{BP1b}}.

\section{Matrix setting}
\label{M}
The following Proposition~\ref{general_discretization} from  \cite{FG} guarantees that for any subspace $X_N$ we can find a sufficiently large number $m_0\in\mathbb{N}$
such that $X_N\in\LD(m_0,p,K_1)$ and $X_N\in\RD(m_0,p,K_2)$.

\begin{oldProposition}[{\cite[Proposition~6.1]{FG}}]\label{general_discretization}
Let $1\leq p<\infty$ and let $X_N\subset L_p(\Omega,\mu)$ be an $N$-dimensional
subspace, where $(\Omega, \mu)$ is a probability space. Then for each
$\varepsilon>0$ there exists $m_0\in\mathbb{N}$ such that if $M\geq m_0$ and
$\{\xi^j\}_{j=1}^M\subset\Omega$ are independent random samples then with
probability at least $(1-\varepsilon)$ we have that $$
(1-\varepsilon)\|f\|_p^p\leq\frac{1}{M}\sum\limits_{j=1}^M|f(\xi^j)|^p\leq (1+\varepsilon)\|f\|_p^p, \quad \forall f\in X_N.
$$
\end{oldProposition}

By Proposition~\ref{general_discretization} the problem of finding the smallest $m$, for which $X_N\in \LD(m,p,D)$ or
$X_N\in\RD(m,p,D)$, can be reformulated in a matrix setting. 
 Note that results similar to Proposition~\ref{general_discretization} with extra conditions on 
 $X_N$ were obtained in \cite[Theorem~2.2]{DPSTT2}. 

Assume that $X_N\in\LD(m_0,p,K_1)$:
$$
\|f\|_p \le K_1 \|S(f,\xi)\|_p,\quad \xi=\{\xi^j\}_{j=1}^{m_0}, \quad \forall f\in X_N.
$$
Consider the space
\be\label{tilde_X_N}
\tilde{X}_N := \{S(f,\xi)\colon f\in X_N\}
\ee
of sampling vectors in $\R^{m_0}$ with the $L_p^{m_0}$-norm. If we obtain an LDI for this space:
\be\label{LDI_X_discr}
\left(\frac{1}{m_0}\sum_{j=1}^{m_0}|y_j|^p\right)^{1/p} \le D \left(\frac{1}{m}\sum_{j=1}^m |y_{i_j}|^p\right)^{1/p},\quad
\forall\by=(y_1,\ldots,y_{m_0})\in\tilde{X}_N,
\ee
then we will get an LDI for the original space:
$$
\|f\|_p \le K_1 \|S(f,\xi)\|_p \le K_1 D \left(\frac{1}{m}\sum_{j=1}^m
|f(\xi^{i_j})|^p\right)^{1/p},\quad \forall f\in X_N,
$$
hence $X_N\in\LD(m,p,K_1D)$.

The RDI case is analogous. If $X_N\in\RD(m_0,p,K_2)$,
\be\label{general_discr_right}
\|S(f,\xi)\|_p \le K_2 \|f\|_p,\quad \xi=\{\xi^j\}_{j=1}^{m_0},  {\quad \forall f\in X_N},
\ee
then an RDI for the space $\tilde{X}_N$ provides an RDI for the original space.

Therefore, it is enough to study the LDI and the RDI problems for discrete systems, which
can be formulated in the matrix setting.

It is pointed out in the survey \cite[{\bf M.1.}]{KKLT} that the right
inequality in the Marcinkiewicz-type discretization theorem (which is the RDI)
for an arbitrary finite-dimensional subspace of $L_2$ follows from the paper \cite{Lun} by
A.\,A.~Lunin.  We show below (see Proposition~\ref{Prop_Lunin}) how to prove
this.  Recall  that  for spaces satisfying the Nikol'skii inequality
\eqref{I4}  {with $p=2$, $q=\infty$} both the RDI and the LDI were obtained in \cite{LimT} using deep results from the paper \cite{MSS}.

It is convenient to write the space of sampling vectors in
the form $\tilde{X}_N = \{A\bx\colon\bx\in\R^N\}$ with an appropriate $m_0\times
N$ matrix $A$. There is a natural way to obtain such a matrix.
Assume that a system $\{u_i\}_{i=1}^N$ forms a basis of $X_N$. Let $A$ be
the matrix with the columns $(u_i(\xi^1), \ldots,
u_i(\xi^{m_0}))^T$, $i=1,\dots,N$.  Then for $f=x_1u_1+\dots+x_Nu_N$,
$\bx=(x_1,\ldots,x_N)\in\R^N$,
 we have
$S(f,\xi) = A\bx$.

Let $A_1$ be an $m\times N$ submatrix of $A$ that consists of rows with some indices
$i_1,\ldots,i_m$ (here we assume that $i_j$ are distinct).
Then \eqref{LDI_X_discr} takes the following form: 
\be\label{matrix_LDI}
\|A\bx\|_p\le D \|A_1\bx\|_p, \quad \forall \bx=(x_1,\dots, x_N)\in \R^N.
\ee
Thus, to get LDI for a discrete system it is sufficient to find a
row-submatrix $A_1$ that satisfies inequality~\eqref{matrix_LDI}, called a
pointwise estimate {(see \cite[Section~3]{KKLT})}.
The only difference between the discretization problem~\eqref{LDI_X_discr} and the
matrix pointwise estimate~\eqref{matrix_LDI} is that submatrices consist of distinct rows $i_j$ and
in the discretization we allow the repetition of $i_j$.
  
Similarly, an RDI for a discrete system is essentially equivalent to the problem of finding a
row-submatrix $A_1$ of the smallest possible size for which the following pointwise estimate  holds  
\be\label{matrix_RDI}
 \|A_1\bx\|_p\le D\|A\bx\|_p, \quad \forall \bx=(x_1,\dots, x_N)\in \R^N.
\ee
{We stress that in~\eqref{matrix_LDI} and~\eqref{matrix_RDI} we consider discrete $L_p^m$-norms~\eqref{discrete_norm} which are
related to the usual $\ell_p^m$-norms by $\|\bz\|_p =
m^{-1/p}\|\bz\|_{\ell_p^m}$.}
 
\begin{Proposition}\label{Prop_Lunin} Let $X_N$ be an $N$-dimensional subspace of $L_2(\Omega,\mu)$, then  $X_N \in
    \RD(N,2,C)$ for an absolute constant $C>0$. 
\end{Proposition}
\begin{proof} By Proposition~\ref{general_discretization}
 inequality~\eqref{general_discr_right} holds for $p=2$ with some $m_0$ and $K_2=2$. 
    W.l.o.g. we may assume that $\dim\tilde{X}_N=N$, see \eqref{tilde_X_N} (if the dimension is
    less than $N$, we will obtain even stronger RDI~--- with less number of
    points).
    Let $\tilde{X}_N = \{A\bx\colon \bx\in\R^N\}$. One can take a matrix $A$
    with columns that are orthonormal in $L_2^{m_0}$.
    In order to obtain RDI we will prove~\eqref{matrix_RDI}.
By Theorem~2 in~\cite{Lun}  there is an $N\times N$ submatrix $A_1$ of the matrix $A$ such that 
\begin{equation*}
    \sup_{\|\bx\|_2=1}\|A_1\bx\|_2\le C \sup\limits_{\|\bx\|_2=1}\|A\bx\|_2,
\end{equation*}
for some absolute constant $C>0$.
    Using that in our case
    $\|A\bx\|_2=(x_1^2+\ldots+x_N^2)^{1/2}=N^{1/2}\|\bx\|_2$ for any
    $\bx\in\R^N$, we obtain:
\begin{equation*}
    \|A_1\bx\|_2\le \|\bx\|_2 \cdot CN^{1/2} = C\|A\bx\|_2.
\end{equation*}
Therefore, $X_N\in\RD(N,2,2C)$.
\end{proof}

\begin{Corollary}\label{Prop_Lunin_even_q}
    Let $p\in\N$ be an even integer, then $X_N \in \RD(N^{p/2},p,C^{2/p})$.
\end{Corollary}

Define the operator $(r,p)$--norm of a matrix $A$ of the size $M\times N$ by
$
  \|A\|_{(r,p)}=\sup\limits_{\|\bx\|_{\ell_r^N}\leq 1}\|A\bx\|_{\ell_p^M}.
  $
We note that~\eqref{matrix_RDI} implies that
\be\label{ineq_for_matrix_norms}
\|A_1\|_{(r,p)}^p\le D^p \frac{m}{m_0}\|A\|_{(r,p)}^p, \quad  \text{ for } r\ge 1.
\ee
The first results of that type were obtained by B.\,S.~Kashin in 1980 on the operator $(2,q)$--norms of matrices for $1\le q\le 2$
(see \cite[Section~3]{KKLT}  for details). If we know that for any $m\times N$
submatrix $A_1$ \eqref{ineq_for_matrix_norms} is not true, then
\eqref{matrix_RDI} does not hold. So there is a chance that there will be
negative results on the norms of submatrices that could help one to establish lower bounds on $m$ in discretization.


\section{An application to sampling recovery}
\label{B}

We begin with a general problem of the recovery of functions in some class and
describe some characteristics of optimal recovery.

Recall the notion of the sampling operator~\eqref{sampling}. Given
a fixed $m$ and a set of points $\xi^1,\ldots,\xi^m\in\Omega$, we 
associate with a function $f\in \C(\Omega)$ the vector
$$
S(f,\xi) := (f(\xi^1),\dots,f(\xi^m)) \in \bbC^m.
$$

Each linear operator $\Phi\colon\mathbb C^m\to L_p(\Omega,\mu)$ specifies a linear algorithm of sampling
recovery:
$$
f\approx \Phi(S(f,\xi)).
$$

For a class of functions $\bF\subset L_p(\Omega,\mu)$ set
$$
\varrho_m(\bF,L_p) := \inf_{\xi;\text{linear}\, \Phi}\,\, \sup_{f\in \bF}
\|f-\Phi(S(f,\xi))\|_p.
$$
This recovery procedure is linear; the following modification of it is also of interest.
We allow any mapping $\Phi\colon \bbC^m \to X_m \subset L_p(\Omega,\mu)$
where $X_m$ is a linear subspace of dimension $m$, and set
$$
\varrho_m^*(\bF,L_p) := \inf_{\xi;\Phi;X_m} \sup_{f\in \bF}\|f-\Phi(S(f,\xi))\|_p.
$$

In both cases above we build an approximant, which comes from a linear subspace of dimension at most $m$.
It is natural to compare quantities $\varrho_m(\bF,L_p)$ and $\varrho_m^*(\bF,L_p)$ with 
the Kolmogorov widths of $\bF$ in $L_p$:
$$
d_m(\bF, L_p) := \inf_{X_m\subset L_p} \sup_{f\in \bF} \inf_{g\in X_m}
\|f-g\|_p.
$$
In the definition of
 Kolmogorov widths we take for a given $f\in \bF$ an approximating element
from $X_m$, which is the element of best
approximation. This means that in general (i.e. if $p\neq 2$) this method of approximation is nonlinear.

We have the following obvious inequalities
\begin{equation*}
d_m (\bF, L_p)\le \varrho_m^*(\bF,L_p)\le \varrho_m(\bF,L_p).
\end{equation*}

 The characteristics $\varrho_m$, $\varrho_m^*$, and their variants are well studied for many particular classes of functions. For an exposition of known results we refer to the books 
\cite{TWW}, \cite{NoLN}, \cite{DTU}, \cite{VTbookMA}, \cite{NW1}--\cite{NW3} and references therein. The characteristics $\varrho_m^*$ and $\varrho_m$ are inspired by the concepts of 
the Kolmogorov width and the linear width. Probably,  $\varrho_m^*$  was introduced in \cite{Di90}, $\varrho_m$ in \cite{VT51}, and a variant of $\varrho_m^*$ without the condition of mapping into a linear subspace  was introduced in \cite{TWW}.

We turn to concrete methods of recovery. Throughout the rest of the section
$\Omega$ is a compact subset of $\mathbb R^d$; we recover continuous functions
$f\in C(\Omega)$ and let $X_N$ denote an $N$-dimensional subspace of
$C(\Omega)$.

We also consider the discrete norms
$$
\|S(f,\xi)\|_p:= \left(\frac{1}{m}\sum_{j=1}^m |f(\xi^j)|^p\right)^{1/p},\quad 1\le p<\infty,
$$
and $\|S(f,\xi)\|_\infty := \max_{j}|f(\xi^j)|$.

For a positive weight $\bw:=(w_1,\dots,w_m)\in \R^m$ consider the following seminorm
$$
\|S(f,\xi)\|_{p,\bw}:= \left(\sum_{j=1}^m w_j |f(\xi^j)|^p\right)^{1/p},\quad 1\le p<\infty.
$$
Define the best approximation of $f\in L_p(\Omega,\mu)$, $1\le p\le \infty$, by elements of $X_N$ as follows
$$
d(f,X_N)_p := \inf_{u\in X_N} \|f-u\|_p.
$$

Theorem~\ref{BT1} below was proved in \cite{VT183} under the following assumptions.

{\bf A1. Discretization.} Let $1\le p\le \infty$. Suppose that
$\xi:=\{\xi^j\}_{j=1}^m\subset \Omega$ provides the WLDI, i.e. for any 
$u\in X_N$ in the case $p<\infty$ we have
$$
\|u\|_p \le D \|S(u,\xi)\|_{p,\bw}  
$$
and in the case $p=\infty$ we have
$$
\|u\|_\infty \le D \|S(u,\xi)\|_{\infty}  
$$
with some positive constant $D$.

{\bf A2. Weights.} Suppose that there is a positive constant
$W$ such that 
$\sum_{j=1}^m w_j \le W$.

Consider the following well known recovery operator (algorithm) 
$$
\ell p\bw(\xi)(f) := \ell p\bw(\xi,X_N)(f):=\text{arg}\min_{u\in X_N}
\|S(f-u,\xi)\|_{p,\bw},\quad 1\le p<\infty,
$$
$$
\ell \infty(\xi)(f) := \ell \infty(\xi,X_N)(f):=\text{arg}\min_{u\in X_N}
\|S(f-u,\xi)\|_{\infty}
$$
(see, for instance, \cite{CM}).
Note that the above algorithm $\ell p\bw(\xi)$ only uses the function values $f(\xi^j)$, $j=1,\dots,m$. In the case $p=2$ it is a linear algorithm -- orthogonal projection with respect 
to the  {seminorm} $\|\cdot\|_{2,\bw}$. Therefore, in the case $p=2$ approximation error by the algorithm $\ell 2\bw(\xi)$ gives an upper bound for the recovery characteristic $\ro_m(\cdot, L_2)$.
 In the case $p\neq 2$ approximation error by the algorithm $\ell p\bw(\xi)$ gives an upper bound for the recovery characteristic $\ro_m^*(\cdot, L_p)$.

\begin{oldTheorem}[{\cite[Theorem~2.1]{VT183}}]\label{BT1}  Under assumptions {\bf A1} and {\bf A2} for any $f\in \C(\Omega)$ we have
for $1\le p<\infty$
$$
\|f-\ell p\bw(\xi)(f)\|_p \le (2DW^{1/p} +1)d(f, X_N)_\infty.
$$
Under assumption {\bf A1} for any $f\in \C(\Omega)$
$$
\|f-\ell \infty(\xi)(f)\|_\infty \le (2D+1)d(f, X_N)_\infty.
$$
\end{oldTheorem}

We now prove a version of Theorem~\ref{BT1} for the error of $\|f-\ell
p\bw(\xi)(f)\|_{\infty}$ under an extra condition of the validity of the Nikol'skii inequality.

\begin{Theorem}\label{BT1a} Let $1\le p<\infty$.  Under assumptions {\bf A1},
    {\bf A2}, and the extra assumption $X_N\in NI(p,\infty,M)$   for any $f\in \C(\Omega)$ we have
 $$
\|f-\ell p\bw(\xi)(f)\|_{\infty} \le (2MD W^{1/p} +1)d(f, X_N)_\infty.
$$
 
\end{Theorem}

\begin{proof} The proof is simple and goes along the lines of the proof of Theorem~\ref{BT1}. Let $u:= \ell p\bw(\xi)(f)$. For an arbitrary $g~\in~X_N$ we have the following chain of inequalities.
$$
\|f-u\|_\infty \le \|f-g\|_\infty + \|g-u\|_\infty \le \|f-g\|_\infty +M\|g-u\|_p
$$
$$
\le  \|f-g\|_\infty + MD\|S(g-u,\xi)\|_{p,\bw}
$$
$$
\le  \|f-g\|_\infty + MD(\|S(f-g,\xi)\|_{p,\bw}+ \|S(f-u,\xi)\|_{p,\bw})
$$
$$
\le  \|f-g\|_\infty + 2MD\|S(f-g,\xi)\|_{p,\bw}
$$
$$
\le  \|f-g\|_\infty + 2MDW^{1/p}\|S(f-g,\xi)\|_{\infty} \le (1+ 2MDW^{1/p})\|f-g\|_\infty.
$$
Minimizing over $g\in X_N$ we complete the proof. 

\end{proof}

Here we prove the following analog of Theorem~\ref{BT1} {under weaker assumption \eqref{B1} instead of {\bf A1}}. 

\begin{Theorem}\label{BT2} Let $p\in [1,\infty)$. Assume that a subspace
    $X_N\subset \C(\Omega)$ has the property $X_N\in\LD(m,p,\infty,D)$ provided
    by a set $\xi=\{\xi^j\}_{j=1}^m$: for each $u\in X_N$
\be\label{B1}
\|u\|_p \le D\max_{1\le j\le m} |u(\xi^j)|.
\ee
Then for any $f\in \C(\Omega)$ we have
$$
\|f-\ell \infty(\xi)(f)\|_p \le (2D  +1)d(f, X_N)_\infty.
$$
\end{Theorem}

\begin{proof}  
 Let $h\in X_N$ be  {an element of} the best $L_\infty$-approximation to $f$ from $X_N$. We have
\be\label{B2}
\|f-h\|_p \le \|f-h\|_\infty = d(f, X_N)_\infty.
\ee
Clearly,
\be\label{B2'}
\|S(f-h,\xi)\|_\infty \le \|f-h\|_\infty = d(f, X_N)_\infty.
\ee
 By the definition of the algorithm $\ell \infty(\xi)$  we obtain
\be\label{B3}
\|S(f-\ell \infty(\xi)(f),\xi)\|_{\infty} \le \|S(f-h,\xi)\|_{\infty} \le  d(f, X_N)_\infty.
\ee
Bounds (\ref{B2'}) and (\ref{B3}) imply that
\begin{equation*}
\|S(h-\ell \infty(\xi)(f),\xi)\|_{\infty} \le 2 d(f, X_N)_\infty.
\end{equation*}
The discretization assumption (\ref{B1}) implies that
\be\label{B5}
\|h-\ell \infty (\xi)(f)\|_{p} \le   2D d(f, X_N)_\infty.
\ee
Combining bounds (\ref{B2}) and (\ref{B5}) we conclude that
$$
\|f-\ell \infty(\xi)(f)\|_p \le (1+2D )  d(f, X_N)_\infty,
$$
which completes the proof.
\end{proof}

Theorems~\ref{BT1} and \ref{BT2} provide the following inequalities for any compact subset $\bF \subset \cC(\Omega)$ and for any probability measure $\mu$ on $\Omega$
\be\label{B6}
\ro_{m}^*(\bF,L_p(\Omega,\mu)) \le Cd_N(\bF,L_\infty), {\quad C=2DW^{1/p}+1}, \quad 1\le p<\infty.
\ee
Here, $m$ is such that in the case of Theorem~\ref{BT1} conditions {\bf A.1} and {\bf A.2} and in the case of Theorem~\ref{BT2} condition~(\ref{B1}) are satisfied for any $N$-dimensional subspace of 
$ \cC(\Omega)$. In the case $p=2$ the following inequality is known (see \cite{VT183}): There exist two positive absolute constants $b$ and $B$ such that for any   compact subset $\Omega$  of $\R^d$, any probability measure $\mu$ on it, and any compact subset $\bF$ of $\cC(\Omega)$ we have
\be\label{R1}
\ro_{bn}(\bF,L_2(\Omega,\mu)) \le Bd_n(\bF,L_\infty).
\ee
It is known (see  {\cite[{\bf D.20. A lower bound}]{KKLT}}) and it follows from
Corollary~\ref{RIC1} that 
for discretization of the $L_p$-norm, $2<p<\infty$, of functions in $N$-dimensional subspace we need at least $m> CN^{p/2}$ points. Thus, we can expect that the way of using an argument similar to Theorem~\ref{BT1} will not allow us to obtain better than  $m$ of order $N^{p/2}$ in the inequality (\ref{B6}) for $2<p<\infty$. Condition (\ref{B1}) is weaker than the combination of conditions {\bf A.1} and {\bf A.2}. Therefore, there is a hope that we can get a better than  $m$ of order $N^{p/2}$ in the inequality (\ref{B6}) by using Theorem~\ref{BT2} if we solve {\bf Open problem \ref{LI}.1}. 

{\bf Open problem \ref{B}.1.} Let $2<p<\infty$. What is the minimal growth of $m(N)$ on $N$, which guarantees the property: For any compact subset $\bF \subset \cC(\Omega)$ and for any probability measure $\mu$ on $\Omega$
$$
\ro_{m(N)}^*(\bF,L_p(\Omega,\mu)) \le Cd_N(\bF,L_\infty),\quad 2<p<\infty.
$$

{\bf Open problem \ref{B}.2.} Let $2<p<\infty$. What is the minimal growth of $m(N)$ on $N$, which guarantees the property: For any compact subset $\bF \subset \cC(\Omega)$ and for any probability measure $\mu$ on $\Omega$
$$
\ro_{m(N)}(\bF,L_p(\Omega,\mu)) \le Cd_N(\bF,L_\infty),\quad 2<p<\infty.
$$

Note that it is known from \cite{VT183} (see \eqref{R1}) that in the case $p=2$ (and, therefore, for all $1\le p\le 2$) one can take $m(N)\le CN$. 

Theorem~\ref{BT2} and  Proposition~\ref{AP6} for $p=2$ imply the following unconditional result.

\begin{Theorem}\label{BT3}There exist two absolute positive constants $C_1$ and $C_2$ 
such that for any $N$-dimensional subspace $X_N \subset \cC(\Omega)$ there exists a set 
of points $\{\xi^j\}_{j=1}^m$, $m\le C_1N$, with the property: For any $f\in \C(\Omega)$ we have
$$
\|f-\ell \infty(\xi,X_N)(f)\|_2 \le (2C_2  +1)d(f, X_N)_\infty.
$$
\end{Theorem}

Note that Proposition~\ref{AP6}  and Theorem~\ref{BT1} with $p=2$ imply the following result.

\begin{Theorem}\label{BT4}
    There exist absolute positive constants $C_1$, $C_2$
such that for any $N$-dimensional subspace $X_N \subset \cC(\Omega)$ there exist a set 
of points $\{\xi^j\}_{j=1}^m$, $m\le C_1N$, with the property: For any $f\in \C(\Omega)$ we have
$$
\|f-\ell 2\bw_m
 (\xi,X_N)(f)\|_2 \le  C_2d(f, X_N)_\infty, \quad \bw_m:=(1/m,\dots, 1/m).
$$
\end{Theorem}

\paragraph{A comment on the sampling recovery in the uniform norm.}
It is well known that results on sampling discretization in the $L_2$-norm imply some results on sampling discretization in the $L_\infty$-norm. We will illustrate this phenomenon on some known examples. 
Probably, the first example of this type is the multivariate trigonometric polynomials. 
By $Q$ we denote a finite subset of $\Z^d$, and $|Q|$ stands for the number of elements in $Q$. Let
$$
\Tr(Q):= \left\{f: f=\sum_{\bk\in Q}c_\bk e^{i(\bk,\bx)},\  \  c_{\bk}\in\mathbb{C}\right\}.
$$
The following theorem was proved in \cite{VT158}. 

\begin{oldTheorem}[{\cite[Theorem~1.1]{VT158}}] \label{TrD}There are three positive absolute constants $C_1$, $C_2$, and $C_3$ with the following properties: For any $d\in \N$ and any $Q\subset \Z^d$   there exists a set of  $m \le C_1|Q| $ points $\xi^j\in \T^d$, $j=1,\dots,m$ such that for any $f\in \Tr(Q)$
    we have
    $$
    C_2\|f\|_2^2 \le \frac{1}{m}\sum_{j=1}^m |f(\xi^j)|^2 \le C_3\|f\|_2^2.
    $$
\end{oldTheorem}

In \cite{DPTT} it was shown how Theorem~\ref{TrD} implies a result on sampling discretization of the $L_\infty$-norm. Namely, the following theorem was proved.

\begin{oldTheorem}[{\cite[Theorem~2.9]{DPTT}}] \label{TrDi} Let two positive absolute constants $C_1$ and $C_2$ be from Theorem~\ref{TrD}. Then for any $d\in \N$ and any $Q\subset \Z^d$   there exists a set $\xi$ of  $m \le C_1|Q| $ points $\xi^j\in \T^d$, $j=1,\dots,m$, such that for any $f\in \Tr(Q)$
we have
$$
\|f\|_\infty \le C_2^{-1/2}|Q|^{1/2}\left(\frac{1}{m}\sum_{j=1}^m |f(\xi^j)|^2\right)^{1/2}\le  C_2^{-1/2}|Q|^{1/2}\max_{1\le j\le m}|f(\xi^j)|.
$$
 \end{oldTheorem}
\begin{proof} For the reader's convenience we present the one line proof from \cite{DPTT} here. We use the set of points provided by Theorem~\ref{TrD}. Then $m\le C_1|Q|$ and for any $f\in\Tr(Q)$ we have
\begin{align*}
\|f\|_\infty & \le |Q|^{1/2}\|f\|_2 \le |Q|^{1/2} C_2^{-1/2} \left(\frac{1}{m}\sum_{j=1}^m |f(\xi^j)|^2\right)^{1/2} \\ & \le |Q|^{1/2} C_2^{-1/2} \max_{1\le j\le m}|f(\xi^j)|.
\end{align*}
\end{proof}

We point out that in the above proof in addition to the sampling discretization result Theorem~\ref{TrD} the Nikol'skii inequality $\|f\|_\infty \le  |Q|^{1/2}\|f\|_2$ has been used. 

The following two sampling discretization results show that the Nikol'skii inequality guarantees good discretization inequalities for general subspaces.

\begin{oldTheorem}[{\cite[Theorem~1.1]{LimT}}]\label{LimTT} Let  $\Omega\subset \R^d$ be a   nonempty set with  the probability measure $\mu$. Assume that a subspace $X_N \in NI(2,\infty, tN^{1/2})$.
Then there is an absolute  constant $C_1$ such that there exists a set $\{\xi^j\}_{j=1}^m\subset \Omega$ of $m \le C_1 t^2 N$ points with the property:
 For any $f\in X_N$  we have  
\begin{equation*}
C_2 \|f\|_2^2 \le \frac{1}{m}\sum_{j=1}^m |f(\xi^j)|^2 \le C_3 t^2\|f\|_2^2, 
\end{equation*}
where $C_2$ and $C_3$ are absolute positive constants. 
\end{oldTheorem}

\begin{oldTheorem}[{\cite[Theorem~1.2]{LimT}}]\label{LimTTw}
    If $X_N$ is an $N$-dimensional subspace of the complex $L_2(\Omega,\mu)$,
    then there exist three absolute positive constants $C_1'$, $c_0'$, $C_0'$,
    a set of $m\leq   C_1'N$ points $\xi^1,\ldots, \xi^m\in\Omega$, and a set of
    nonnegative  weights $\lambda_j$, $j=1,\ldots, m$,  such that
    $$
    c_0'\|f\|_2^2\leq  \sum_{j=1}^m \lambda_j |f(\xi^j)|^2 \leq  C_0'
    \|f\|_2^2,\quad   \forall f\in X_N.
    $$
\end{oldTheorem}

Note that Theorem~\ref{LimTTw} is a generalization to the complex case of an earlier result from \cite{DPSTT2} established for the real case. 
In \cite{KKT} it was shown how Theorem~\ref{LimTTw} implies a result on sampling discretization of the $L_\infty$-norm.
Namely, the following theorem was proved there.

\begin{oldTheorem}[{\cite[Theorem~6.6]{KKT}}]\label{BP1}
    There exist two absolute constants $C_1$ and $C_2$ such that for any subspace
$X_N \in NI(2,\infty, M)$  there exists a set of $m\leq C_1N$ points
$\xi^1,\ldots, \xi^m\in\Omega$ with the property: for any $f\in X_N$ we have
\begin{equation*}
\|f\|_\infty \le C_2 M \max_{1\le j\le m} |f(\xi^j)|.
\end{equation*}
\end{oldTheorem} 

Note that in the special case when $M=tN^{1/2}$ the proof of Theorem~\ref{BP1}
from \cite{KKT} shows that Theorem~\ref{LimTT} implies the following result on
simultaneous sampling discretization of the $L_\infty$
and $L_2$ norms.

\begin{Theorem}\label{BP1a} Let  $\Omega\subset \R^d$ be a   nonempty set with
    the probability measure $\mu$. Assume that $X_N \in \NI(2,\infty, tN^{1/2})$.
    Then there are three positive absolute  constants $C_1$, $C_2$, $C_3$ (they
    are from Theorem~\ref{LimTT}) such that there exists a set
    $\{\xi^j\}_{j=1}^m\subset \Omega$ of cardinality $m \le C_1 t^2 N$ with the
    property:
     for any $f\in X_N$  we have  
     \be\label{conddisc}
     \|f\|_\infty \le C_2^{-1/2}tN^{1/2}\left(\frac{1}{m}\sum_{j=1}^m
     |f(\xi^j)|^2\right)^{1/2} \le C_2^{-1/2}tN^{1/2} \max_{1\le j\le m}
     |f(\xi^j)|.
     \ee
     Moreover, for any $f\in X_N$  we have
     \begin{equation*}
     C_2 \|f\|_{L_2(\Omega,\mu)}^2 \le \frac{1}{m}\sum_{j=1}^m |f(\xi^j)|^2 \le
     C_3 t^2\|f\|_{L_2(\Omega,\mu)}^2. 
    \end{equation*}
\end{Theorem} 

The above Theorems~\ref{BP1} and~\ref{BP1a} are conditional results -- they hold under 
the Nikol'skii inequality assumption. Recently, it was understood how those
conditional results can be converted into unconditional ones (see \cite{KPUU2}).
This progress was
based on the result of J.~Kiefer and J.~Wolfowitz \cite{KW}, which
 guarantees that for any finite
dimensional subspace $X_N$ of $\cC(\Omega)$ there exists a probability measure
$\mu$ on $\Omega$ such that for all $f\in X_N$ we have
\be\label{KW}
\|f\|_\infty \le N^{1/2}\|f\|_{L_2(\Omega,\mu)}.
\ee
In other words, for any subspace $X_N$ of $\cC(\Omega)$ we have $X_N \in
NI(2,\infty, N^{1/2})$ with some probability measure $\mu$.

 It was observed in \cite{KPUU2} (see Remark 6 there) that the "Kiefer-Wolfowitz result presents itself as a missing piece" in known results proved under the Nikol'skii inequality assumption  $X_N \in
NI(2,\infty, CN^{1/2})$.    The following result was proved in \cite{KPUU2}. 

\begin{oldTheorem}[{\cite[Theorem 2]{KPUU2}}]\label{BP1b}
There is an absolute  constant $C$   such that for any subspace $X_N$ of
    $\cC(\Omega)$ there exists a set $\{\xi^j\}_{j=1}^m\subset \Omega$ of $m \le
    2 N$ points with the following property:
 for any $f\in X_N$  we have  
\be\label{conddisc1}
\|f\|_\infty \le CN^{1/2}\left(\frac{1}{m}\sum_{j=1}^m |f(\xi^j)|^2\right)^{1/2}.
\ee
\end{oldTheorem} 

The authors of~\cite{KPUU2} proved~(\ref{conddisc1})
using the Kiefer--Wolfowitz result~\eqref{KW} (they also proved a novel version of it which is actually WLDI($\infty$, $2$))
and LDI(2,2), see~\eqref{ldi_2_2}, with $m=2N$ points.
They also mention that the LDI($\infty,2$) property \eqref{conddisc1} implies
the following LDI($\infty,\infty$):
\be\label{LDI_infty}
\|f\|_\infty \le C_2 N^{1/2}\max_{1\le j\le m} |f(\xi^j)|, \quad m\leq C_1N.
\ee
Note that
inequality \eqref{LDI_infty} was given in \cite{KoTe} without proof.   
It is a direct corollary of Theorem~\ref{BP1} and the 
Kiefer-Wolfowitz result.

\begin{Remark}\label{BPRem} The Kiefer--Wolfowitz result \eqref{KW} makes the first half  of Theorem~\ref{BP1a} (inequalities \eqref{conddisc}) with $t=1$ unconditional (it holds for any $X_N\subset \cC(\Omega)$) because \eqref{conddisc} does not depend on measure $\mu$.   In addition, Theorem \ref{BP1a} (with $t=1$) provides the following useful discretization of the $L_2$ norm for $\mu$ satisfying $X_N \in NI(2,\infty, N^{1/2})$ (in particular, for $\mu$ from \eqref{KW})
\begin{equation*}
  C_2 \|f\|_{L_2(\Omega,\mu)}^2 \le \frac{1}{m}\sum_{j=1}^m |f(\xi^j)|^2 \le C_3 \|f\|_{L_2(\Omega,\mu)}^2.
\end{equation*}
\end{Remark}

Let us now discuss the sampling recovery in the uniform norm.

Inequality~\eqref{LDI_infty} provide 
the discretization assumption {\bf A1} in the case $p=\infty$ with $D = C_2N^{1/2}$. 
Therefore, Theorem~\ref{BT1} implies that for any subspace $X_N$ of $\cC(\Omega)$ 
we have for any $f\in \cC(\Omega)$
$$
\|f-\ell \infty(\xi)(f)\|_\infty \le (2C_2N^{1/2}  +1)d(f, X_N)_\infty.
$$
In turn, this inequality implies that for any compact subset $\bF \subset \cC(\Omega)$ we have
\be\label{Bnonl}
\ro_{bn}^*(\bF,L_\infty) \le Bn^{1/2}d_n(\bF,L_\infty)
\ee
with positive absolute constants $b$ and $B$. 

We now apply Theorem~\ref{BT1a} in the case $p=2$. We use Theorem~\ref{LimTTw} and
Remark~\ref{rem_marcink_sum_weights} to satisfy the assumptions {\bf A1}, {\bf A2} and use the Kiefer-Wolfowitz
result (\ref{KW}) to satisfy the Nikol'skii inequality assumption. As a result, instead 
of (\ref{Bnonl}) we obtain its linear version
\be\label{Blin}
\ro_{bn}(\bF,L_\infty) \le Bn^{1/2}d_n(\bF,L_\infty).
\ee
Note that the inequality (\ref{Blin}) was formulated in \cite{KPUU}. 
We point out that the inequalities (\ref{Bnonl}) and (\ref{Blin}) are the direct corollaries of 
the known results, which are based on the sampling discretization results of the $L_2$-norm 
from \cite{LimT}, and the Kiefer-Wolfowitz result (\ref{KW}). 

\begin{Remark}\label{DR1}
    It is pointed out in {\cite[Section~6]{KoTe}} that Theorem~\ref{LimTTw} and  Theorem~\ref{BP1} hold for $m\le bN$
    where $b\in (1,2]$ is arbitrary and $c_0'$, $C_0'$, $C_2$ are allowed to depend on $b$. This implies that
 inequalities (\ref{LDI_infty}), (\ref{Bnonl}) and (\ref{Blin}) hold for each $C_1\in (1,2]$ and $b\in (1,2]$ with $C_2$ allowed to depend on $C_1$ and $B$ allowed to depend on $b$. 
\end{Remark}

\section{Universal discretization and nonlinear recovery} 
\label{UD}

In this section we obtain some results directly related to the very recent results from 
\cite{DTM1}--\cite{DTM3}. 

Recall the definition of the concept of $v$-term approximation with respect to a given system $\D=\{g_i\}_{i=1}^\infty$ in a Banach space $X$.
Given an integer $v\in\N$, we denote by $\mathcal{X}_v(\D)$ the collection of
all linear spaces spanned by $g_j$, $j\in J$  with $J\subset \N$ and $|J|=v$,
and we denote by $\Sigma_v(\D)$ the set of all $v$-term approximants with respect
to $\D$:
$$
\Sigma_v(\D):= \bigcup_{V\in\cX_v(\D)} V.
$$

Let 
$$
\sigma_v(f,\D)_X := \inf_{g\in\Sigma_v(\D)}\|f-g\|_X
$$
denote the best $v$-term approximation of  $f\in X$   in the $X$-norm  with respect to $\D$,   and define 
$$
 \sigma_v(\bF,\D)_X := \sup_{f\in\bF} \sigma_v(f,\D)_X,\quad  \sigma_0(\bF,\D)_X := \sup_{f\in\bF} \|f\|_X
 $$
for a function class $\bF\subset X$. 
  
In this paper we only consider finite systems $\mathcal D_N=\{g_i\}_{i=1}^N$.
  We now describe results obtained in \cite{DTM3}. The following two algorithms were studied 
  in \cite{DTM3}.  Denote 
  $$
  \ell p(\xi,L) := \ell p\bw_m(\xi,L),\qquad \bw_m := (1/m,\dots,1/m).
  $$
  
  {\bf Algorithm $ \ell p$.} Given a system $\D_N$ and a set of points
  $\xi:=\{\xi^j\}_{j=1}^m\subset\Omega$ we define the following algorithm:
$$
L(\xi,f) := \text{arg}\min_{L\in \cX_v(\D_N)}\|f-\ell p(\xi,L)(f)\|_p,
$$
\begin{equation*}
  \ell p(\xi,\cX_v(\D_N))(f):= \ell p(\xi,L(\xi,f))(f).
\end{equation*}

{\bf Algorithm $\ell p^s$.} Given a system $\D_N$ and a set of points
$\xi:=\{\xi^j\}_{j=1}^m\subset \Omega$,
we define the following algorithm:
$$
L^s(\xi,f) := \text{arg}\min_{L\in \cX_v(\D_N)}\|S(f-\ell p(\xi,L)(f),\xi)\|_p,
$$
\begin{equation*}
  \ell p^s(\xi,\cX_v(\D_N))(f):= \ell p(\xi,L^s(\xi,f))(f).
\end{equation*}

Here index $s$ stands for {\it sample} to stress that this algorithm only uses
the sample vector $S(f,\xi)$.

Clearly, $\ell p^s(\xi,\cX_v(\D_N))(f)$ is the best $v$-term approximation of $f$ with respect to $\D_N$ in 
the space $L_p(\xi)$ with the norm $\|S(f,\xi)\|_p$.
So,
\be\label{ub15}
\|f-\ell p^s(\xi,\cX_v(\D_N))(f)\|_{L_p(\xi)} = \sigma_v(f,\D_N)_{L_p(\xi)}.
\ee
Note that $\ell p^s(\xi,\cX_v(\D_N))(f)$ may not be unique. 

In this paper we discuss Algorithm $\ell p^s$ and the following version of Algorithm $ \ell p$.

{\bf Algorithm $ \ell (p,\infty)$.} Given a system $\D_N$ and a set of points
$\xi:=\{\xi^j\}_{j=1}^m\subset \Omega$ we define the following algorithm:
$$
L(\xi,f) := \text{arg}\min_{L\in \cX_v(\D_N)}\|f-\ell \infty(\xi,L)(f)\|_p,
$$
\be\label{Alg3}
  \ell (p,\infty)(\xi,\cX_v(\D_N))(f):= \ell \infty(\xi,L(\xi,f))(f).
\ee

In \cite{DTM3} the following one-sided universal discretization condition on the
collection of subspaces was used.

 \begin{Definition}[\cite{DTM3}]\label{ID1a}
     Let $1\le p <\infty$. We say that a set $\xi:= \{\xi^j\}_{j=1}^m \subset
     \Omega $ provides {\it one-sided $L_p$-universal discretization with
     parameter $D\ge 1$} for a
     collection $\cX:= \{X(n)\}_{n=1}^k$ of finite-dimensional  linear subspaces
     $X(n)$  if we have 
     \be\label{I3a}
     \|f\|_p \le D\left(\frac{1}{m} \sum_{j=1}^m |f(\xi^j)|^p\right)^{1/p} \quad \text{for any}\quad f\in \bigcup_{n=1}^k X(n) .
     \ee
\end{Definition}

In our terms the condition~\eqref{I3a} means that
$\bigcup\limits_{n=1}^k X(n) \in \LD(m,p,D)$
and the points $\{\xi^j\}_{j=1}^m$ provide this property. So we will say that the
set $\xi$ provides universal LDI($p$).

In this paper we use the following weaker condition.

\begin{Definition}\label{UDD1}
    Let $1\le p <\infty$. We say that a set $\xi:= \{\xi^j\}_{j=1}^m \subset
    \Omega $ provides {\it universal LDI$(p,\infty)$  with parameter $D\ge 1$} for the collection $\cX:=
    \{X(n)\}_{n=1}^k$ of finite-dimensional linear subspaces $X(n)$ if $\xi$ provides the property $\bigcup\limits_{n=1}^k
    X(n)\in \LD(m,p,\infty,D)$, i.e.
    \be\label{I3}
    \|f\|_p \le D\max_{1\le j\le m} |f(\xi^j)| \quad \text{for any}\quad f\in \bigcup_{n=1}^k X(n).
    \ee
\end{Definition}

The following two Lebesgue-type inequalities were proved in \cite{DTM3} for Algorithms $\ell p$ and $\ell p^s$. We only formulate those parts of these theorems from \cite{DTM3}, which 
are related to our results. 

 \begin{oldTheorem}\label{ubT3a} Let $m$, $v$, $N$ be natural numbers such
     that $v\le N$.  Let $\D_N\subset \C(\Og)$ be  a system of $N$ elements.
     Assume that  there exists a set $\xi:= \{\xi^j\}_{j=1}^m \subset \Omega $,
     which provides universal LDI$(p)$  property (\ref{I3a}) with $1\le p<\infty$
     for the collection $\cX_v(\D_N)$. Then for   any  function $ f \in
     \C(\Omega)$ we have 
     \begin{equation*}
  \|f-\ell p(\xi,\cX_v(\D_N))(f)\|_p \le  (2D +1) \sigma_v(f,\D_N)_\infty.
\end{equation*}
 \end{oldTheorem}
 
  \begin{oldTheorem}\label{ubT5a} Let $m$, $v$, $N$ be natural numbers such
      that $2v\le N$.  Let $\D_N\subset \C(\Og)$ be  a system of $N$ elements.
      Assume that  there exists a set $\xi:= \{\xi^j\}_{j=1}^m \subset \Omega $,
      which provides the universal LDI$(p)$ property (\ref{I3a})  with $1\le p<\infty$
      for the collection $\cX_{2v}(\D_N)$. Then for   any  function $ f \in
      \C(\Omega)$ we have 
      \begin{equation*}
  \|f-\ell p^s(\xi,\cX_v(\D_N))(f))\|_p \le  (2D +1) \sigma_v(f,\D_N)_\infty.
\end{equation*}
 \end{oldTheorem}

We prove two conditional theorems here: Theorem~\ref{ubT3} for Algorithm $\ell(p,\infty)$ 
and Theorem~\ref{ubT5} for Algorithm $\ell p^s$. 

\begin{Theorem}\label{ubT3}
Let $m$, $v$, $N$ be natural numbers such that
$v\le N$.  Let $\D_N\subset \C(\Og)$ be  a system of $N$ elements. Assume
that  there exists a set $\xi:= \{\xi^j\}_{j=1}^m \subset \Omega $, which
provides the universal LDI$(p,\infty$)  property (\ref{I3}) with $1\le p<\infty$
for the collection $\cX_v(\D_N)$. Then for   any  function $ f \in
\C(\Omega)$ we have
 \begin{equation*}\label{I6}
  \|f-\ell (p,\infty)(\xi,\cX_v(\D_N))(f)\|_p \le  (2D +1) \sigma_v(f,\D_N)_\infty.
\end{equation*}
 \end{Theorem}
\begin{proof}  
Suppose that   a set $\xi:= \{\xi^j\}_{j=1}^m \subset \Omega $ provides
the universal LDI($p,\infty$) for the collection $\cX_v(\D_N)$. Then condition
(\ref{B1}) is satisfied for all $X(n)$ from the collection $\cX_v(\D_N)$
with the same $D$.  Thus, we can apply Theorem~\ref{BT2} with the same set of
    points $\xi$ to each subspace $X(n)$. This yields the inequality
\be\label{A3}
\|f-\ell \infty(\xi,X(n))(f)\|_p \le (2D +1)d(f, X(n))_{\infty}.
\ee
for all $n=1,\dots,\binom{N}{v}$.
 
Then inequality (\ref{A3}) and the definition (\ref{Alg3}) imply that
\begin{equation*}
 \|f-\ell (p,\infty)(\xi,\cX)(f)\|_p \le   (2D +1)\min_{1\le n\le k}d(f, X(n))_{\infty}.
\end{equation*}
This finishes the proof.  
  \end{proof}

We prove the following analogue of Theorem~\ref{ubT3}
for Algorithm $\ell p^s$ that only uses the
function values at the points $\xi^1,\ldots,\xi^m$.

 \begin{Theorem}\label{ubT5}
 Let $m$, $v$, $N$ be natural numbers such
 that $2v\le N$.  Let $\D_N\subset \C(\Og)$ be  a system of $N$ elements.
 Assume that  there exists a set $\xi:= \{\xi^j\}_{j=1}^m \subset \Omega $,
     which provides the universal LDI$(p,\infty$) property (\ref{I3}) with $1\le p<\infty$ for the collection $\cX_{2v}(\D_N)$.
     Then for   any  function $f \in \C(\Omega)$ we have
 \be\label{ub17}
  \|f - \ell p^s(\xi,\cX_v(\D_N))(f)\|_p \le  (2D +1) \sigma_v(f,\D_N)_\infty.
 \ee
 \end{Theorem}
 \begin{proof} We derive  (\ref{ub17}) from (\ref{ub15}).    Clearly, 
$$
\sigma_v(f,\D_N)_{L_\infty(\xi)} \le \sigma_v(f,\D_N )_\infty.
$$
For brevity set $u:=\ell p^s(\xi,\cX_v(\D_N))(f)$ and let $h$ be the best
$L_\infty$-approximation to $f$ from $\Sigma_v(\D_N)$.   Then (\ref{ub15})
implies that
$$
\|h - u\|_{L_\infty(\xi)} \le \|f-h\|_{L_\infty(\xi)} +\|f-u\|_{L_\infty(\xi)} 
 \le 2\sigma_v(f,\D_N)_\infty.
$$
Using that $h - u \in \Sigma_{2v}(\D_N)$, by discretization (\ref{I3}) we 
conclude that
\be\label{ub18}
\|h - u\|_{L_p(\Omega,\mu)} \le 2D \sigma_v(f,\D_N)_\infty.
\ee
Finally,
$$
\|f-u\|_{L_p(\Omega,\mu)} \le \|f-h\|_{L_p(\Omega,\mu)} + \|h - u\|_{L_p(\Omega,\mu)}.
$$
This and (\ref{ub18}) prove (\ref{ub17}).

\end{proof}

We now discuss an application of Theorem~\ref{ubT5}  to optimal sampling recovery.
We allow any mapping $\Psi : \bbC^m \to   L_p(\Omega,\mu)$  and for a function class $\bF\subset \cC(\Omega)$ define
$$
\varrho_m^o(\bF,L_p) := \inf_{\xi } \inf_{\Psi} \sup_{f\in \bF}\|f-\Psi(f(\xi^1),\dots,f(\xi^m))\|_p.
$$
Here, index {\it o} stands for optimal. The following Theorem~\ref{ubT6} is a direct corollary of 
Theorem~\ref{ubT5}.

\begin{Theorem}\label{ubT6}
    Let $m$, $v$, $N$ be natural numbers such that
    $2v\le N$.   Let $1\le p<\infty$. Assume that a system $\D_N \subset \C(\Og)$ is such
    that there exists a set $\xi:= \{\xi^j\}_{j=1}^m \subset \Omega $, which
    provides the universal LDI$(p,\infty$), (\ref{I3}),
    for the collection $\cX_{2v}(\D_N)$.  Then for   any compact subset $\bF$ of
    $\C(\Omega)$
 \begin{equation*}
 \varrho_{m}^{o}(\bF,L_p(\Omega,\mu)) \le  (2D +1) \sigma_{v}(\bF,\D_N)_\infty.
\end{equation*}
 \end{Theorem}
 
\section{Some remarks on the discretization of the uniform norm}
\label{C}
Discretization of the uniform norm is an actively developing area of research, the reader can find recent results in the papers \cite{KKT}, \cite{KoTe}, \cite{VT2023}, \cite{KPUU}. 

Let $\Phi:=\{\ff_i\}_{i=1}^N$ be a real {linearly} independent
system satisfying the condition 
\be\label{CE}
\sum_{i=1}^N \ff_i^2 \le Nt^2.
\ee
For 
$$
f = \sum_{i=1}^N a_i(f)\ff_i
$$
we define the vector 
$$
A(f):= (a_1(f),\dots,a_N(f)) \in \R^N.
$$
Let $X_N := \sp\{\ff_1,\dots,\ff_N\}$ and
$$
B_\Phi(L_\infty):= \{A(f): f\in X_N, \quad \|f\|_\infty \le 1\} \subset \R^N.
$$

We present here some results similar to those from \cite{KaTe03} (see also \cite{VTbookMA}, pp. 344--345). 
We begin with the following conditional statement.
\begin{Theorem}\label{CT1} Assume that the system $\Phi$ satisfies (\ref{CE}) and has the following properties:
\begin{equation}\label{C1}
\vol(B_\Phi(L_\infty))^{1/N} \le KN^{-1/2}, 
\end{equation}
    and $X_N\in \LD(m,\infty, D)$, i.e. {for some set $\xi := \{\xi^j\}_{j=1}^m$}
\begin{equation}\label{C2}
    \forall f \in X_N \qquad \|f\|_\infty\le D\max_{1\le j\le m}|f(\xi^j)|.
\end{equation}
Then there exists an absolute positive constant $c$ such that
$$
m\ge \frac{N}{e}e^{c(t K D)^{-2}}.
$$
\end{Theorem}
\begin{proof} We use the following result of E.~Gluskin \cite{G}. Here
    $\langle\bx^1,\bx^2\rangle=\sum x^1_ix^2_i$ is the usual scalar product and
    $|\bx|=|\langle\bx,\bx\rangle|^{1/2}$ is the {Euclidean} norm.
\begin{oldTheorem}\label{CT2} Let $Y=\{\by^1,\dots,\by^S\} \subset \R^N$, $|\by^j|=1$, $j=1,\dots,S$, $S\ge N$, and 
$$
W(Y) := \{\bx\in \R^N:|\langle\bx,\by^j\rangle| \le 1,\quad j=1,\dots,S\}.
$$
Then
$$
\vol(W(Y))^{1/N} \ge C(1+\ln (S/N))^{-1/2}.
$$
\end{oldTheorem}

By the assumption (\ref{C2}) we have $m\geq N$ and 
\be\label{C2p}
    \{A(f)\colon f\in X_N,\quad |f(\xi^j)|\le D^{-1},\;j=1,\ldots,m\} \subseteq B_\Phi(L_\infty). 
\ee
Further, for any $\omega\in \Omega$
$$
|f(\omega)| = |\langle A(f),\bz(\omega)\rangle|,\qquad \bz(\omega) :=
    (\ff_1(\omega),\dots,\ff_N(\omega)) \in \R^N.
$$
    Note that $|\bz(\omega)|\le(Nt^2)^{1/2}$ by (\ref{CE}).

    It is clear that the condition
$$
    |f(\omega)| \le D^{-1}\quad\mbox{for all $\omega\in\{\xi^j\}_{j=1}^m$}
$$
is satisfied if
$$
|\langle A(f), \bz(\xi^j)\rangle| \le D^{-1}, \quad j=1,\dots,m.
$$
Now let
$$
\by^j := \bz(\xi^j)/|\bz(\xi^j)|,\qquad Y:=\{\by^j,\, j=1,\dots, m\}.
$$
By Theorem~\ref{CT2} with $S=m$ we obtain
\be\label{C3p}
\vol(W(Y))^{1/N} \ge C(1+\ln (S/N))^{-1/2} . 
\ee
The condition
$$
|\langle A(f),\by^j\rangle|\le N^{-1/2}t^{-1}D^{-1}, \quad j=1,\dots,m
$$
implies that
$$
|\langle A(f), \bz(\xi^j)\rangle|=|\langle A(f),\by^j\rangle|\cdot |\bz(\xi^j)|
$$
$$
\le (Nt^2)^{1/2} N^{-1/2}t^{-1}D^{-1} =  D^{-1}, \quad j=1,\dots,m.
$$
Therefore, from (\ref{C2p}) and (\ref{C3p}) we get
$$
\vol(B_\Phi(L_\infty))^{1/N} \ge CN^{-1/2}t^{-1}D^{-1}(1+\ln (m/N))^{-1/2}.
$$
We now use our assumption (\ref{C1}) and conclude that
\begin{equation*}\label{C4p}
tKD \ge C'(\ln (em/N))^{-1/2}.
\end{equation*}
This completes the proof.
\end{proof}

We now discuss a corollary of Theorem~\ref{CT1}.
Let $\Phi:=\{\ff_i\}_{i=1}^N$ be a uniformly bounded
Sidon system:
\be\label{Si}
\bt \sum_{i=1}^N |a_i| \le \|\sum_{i=1}^N a_i\ff_i\|_\infty, \quad \forall \{a_i\}_{i=1}^N;
\quad \|\ff_i\|_\infty \le B,\quad i=1,\dots,N.
\ee

\begin{Corollary}\label{CC1} Assume that the system $\Phi$ satisfies (\ref{Si})
    and $X_N\in\LD(m,\infty,D)$.
Then there exists an absolute positive constant $c$ such that
$$
m\ge \frac{N}{e}e^{c\frac{\bt^2N}{B^2D^2}}.
$$
\end{Corollary}
\begin{proof} Condition (\ref{Si}) implies that the system $\Phi$ satisfies (\ref{CE}) with $t=B$ and 
\begin{equation*}\label{C1c}
\vol(B_\Phi(L_\infty))^{1/N} \le C_1\bt^{-1}N^{-1}  
\end{equation*}
with an absolute constant $C_1$. 
Applying Theorem~\ref{CT1} with $t=B$, $K=C_1\bt^{-1}N^{-1/2}$, and $D$, we complete the proof.
\end{proof}

The Sidon property is connected to other interesting properties of functional
systems. This gives more examples of systems with poor $L_\infty$
discretization.

Recall the subgaussian norm
$$
\|\varphi\|_{\psi_2} := \inf\left\{C>0\colon \int_\Omega\exp(\varphi^2/C^2)d\mu\le 2\right\}.
$$
We call a system $\varphi_1,\ldots,\varphi_n$ subgaussian if $\|\sum_{k=1}^n a_k
\varphi_k\|_{\psi_2}\le
C(\sum_{k=1}^n a_k^2)^{1/2}$ for all coefficients and some fixed $C$.

\begin{Example}
    For any uniformly bounded orthonormal subgaussian system $\{\varphi_i\}_{i=1}^N$ the
    space $X_N := \sp\{\varphi_i\}_{i=1}^N$ requires exponential number of
    points for $L_\infty$-discretization.
\end{Example}

Indeed, it was proven in~\cite[Theorem~1.7]{BL17} that such system has a Sidon subsystem of
size $cN$.

Recall some probabilistic notations: $\mathsf{E}$ means expectation,
and $\mathcal N(0,1)$ is the standard normal distribution.
\begin{Example}
    For any uniformly bounded orthonormal system $\{\varphi_i\}_{i=1}^N$ that is
    randomly Sidon:
    $$
    \mathsf{E} \|\sum_{i=1}^N \mathbf{g}_i a_i \varphi_i \|_\infty \ge \beta\|a\|_1,
    \quad \mathbf{g}_i \sim \mathcal N(0,1),\quad\forall a\in\R^N,
    $$
    the space $X_N := \sp\{\varphi_i\}_{i=1}^N$ requires exponential number of
    points for $L_\infty$-discretization.
\end{Example}

Indeed, it was proven in~\cite{P18} that the random Sidon property for
$\Phi:=\{\varphi_i\}_{i=1}^N$ implies
sidonicity of the system $\Phi^{(4)} :=
\{\varphi_i(x^1)\cdots\varphi_i(x^4)\}_{i=1}^N$.
If some points $\{\xi^j\}_{j=1}^m$ discretize $\sp\Phi$, then the points
$\{(\xi^{j_1},\xi^{j_2})\}$ discretize $\sp\Phi^{(2)}$:
\begin{multline*}
|\sum_{k=1}^N a_k \varphi_k(x)\varphi_k(y)| = |\sum_{k=1}^N
    (a_k\varphi_k(y))\varphi_k(x)| \le D \max_{1\le j\le m}|\sum_{k=1}^N
    (a_k\varphi_k(y))\varphi_k(\xi^j)| = \\
    = D \max_{1\le j\le m}|\sum_{k=1}^N (a_k\varphi_k(\xi^j))\varphi_k(y)| \le
    D^2 \max_{1\le j\le m} \max_{1\le l\le m} |\sum_{k=1}^N a_k\varphi_k(\xi^j)\varphi_k(\xi^l)|.
\end{multline*}
Further, points
$\{(\xi^{j_1},\xi^{j_2},\xi^{j_3},\xi^{j_4})\}$ discretize Sidon system $\sp\Phi^{(4)}$
so the number of points is exponential in $N$.

{\bf Acknowledgements.} The third author is grateful to D. Krieg for bringing the paper \cite{KW} to his
attention.

\Addresses

\end{document}